\documentclass[notitlepage,a4,article,nofootinbib]{revtex4-1}%,twocolumn]{revtex4}

\usepackage{amsmath}%
\usepackage{amsfonts}%
\usepackage{amssymb}%
\usepackage{mathrsfs}
\usepackage{amsthm}
\usepackage{breqn}
\newcommand{\sgn} {\mathrm{sgn}\,}

\newcommand{\mc}{\mathcal}
\newcommand{\cp}{\times}

\newcommand{\di}{\nabla\cdot}
\newcommand{\cu}{\nabla\times} 
\newcommand{\JI}[1]{\bol{#1}\cdot\cu{\bol{#1}}}

\newcommand{\Norm}[1]{\left\lvert\left\lvert #1 \right\rvert\right\rvert}

\newcommand{\bol}{\boldsymbol}
\newcommand{\hb}[1]{\hat{\bol{#1}}}
\newcommand{\abs}[1]{\left\lvert{#1}\right\rvert}
\newcommand{\w}{\wedge}
\newcommand{\lr}[1]{\left({#1}\right)}

\newcommand{\mf}{\mathfrak}
\newcommand{\p}{\partial}

\newcommand{\ov}[1]{\mkern 1.5mu\overline{\mkern-1.5mu#1\mkern-1.5mu}\mkern 1.5mu}
\newtheorem{mydef}{\textit{Def}}[section]
\newtheorem{remark}{\textit{Remark}}[section]
\newtheorem{theorem}{\textit{Theorem}}[section]
\newtheorem{corollary}{\textit{Corollary}}[section]
\newtheorem{proposition}{\textit{Proposition}}[section]
\newtheorem{example}{\textit{Example}}[section]
\newtheorem{lemma}{\textit{Lemma}}[section]
\makeatletter
\let\cat@comma@active\@empty
\makeatother

\begin{document}
\title{Degenerate Laplacian: A Classification by Helicity}
\author{N. Sato}
\affiliation{Research Institute for Mathematical Sciences, Kyoto University, Kyoto 606-8502, Japan}
\author{Z. Yoshida}
\affiliation{Graduate School of Frontier Sciences, The University of Tokyo,
Kashiwa, Chiba 277-8561, Japan}
\date{\today}

\begin{abstract}
We study a class of generalized Laplacian operators by violating the ellipticity with degenerate metric tensors.
The theory is motivated by the statistical mechanics of topologically constrained particles.
In the context of diffusion models, the metric tensor is given by $g=-\mathcal{J}^2$ with a generalized Poisson matrix $\mathcal{J}$ that dictates particle dynamics.
The standard Euclidean metric corresponds to the symplectic matrix of canonical Hamiltonian systems.
However, topological constraints bring about nullity to $\mathcal{J}$,
resulting in degeneracy in the corresponding diffusion operator;
we call such an operator an orthogonal Laplacian 
(since the ellipticity is broken in the direction parallel to the nullity),
and denote it by $\Delta_{\perp}$.
Although all nice properties pertinent to the ellipticity are generally lost for $\Delta_{\perp}$,
a finite helicity of $\mathcal{J}$ helps to recover some of them  by preventing foliation of space.
We show that
$-\lr{\Delta_{\perp}u,v}$ defines an inner product of a Sobolev-like Hilbert space, and satisfies a Poincar\'e-like inequality $-\lr{\Delta_{\perp}u,u}\geq C\Norm{u}^{2}_{L^2}$.
Applying Riesz's representation theorem, we obtain a unique weak solution of the orthogonal Poisson equation.
\end{abstract}


\maketitle

\begin{normalsize}

\section{Introduction}

In order to reflect the effect of a non-trivial geometry of space, the 
classical Laplacian $\Delta = \sum_j \partial_j^2 $ ($\partial_j = 
\partial/\partial x^j$ with Cartesian coordinates $\bol{x}=\lr{x^1,\cdots,x^n}$) is 
generalized to an operator such as
\begin{equation}
\mathscr{L} = %\sum_j 
\partial_j g^{jk} \partial_k,\label{L}
\end{equation}
with some ``metric tensor'' $g^{jk}$. 
%\footnote{
%Here we consider scalar fields. For more general differential forms,
%$\mathscr{L}=\delta d + d \delta$ with the Hodge's dual 
%differential $\delta$.}.
When $g^{jk}$ is a Riemanian metric 
(i.e., $g^{jk}=g^{kj}$ with all positive-definite eigenvalues),
$\mathscr{L}$ is ``elliptic'' being essentially equivalent to $\Delta$ (see Refs. \cite{Gilbarg,Evans,Brezis}).
Here we allow $g^{jk}$ to be ``positive semi-definite'' in the sence that
$g^{jk}q_{j}q_{k} \geq 0$ for any covector ${q}$.
Then, $\mathscr{L}$ becomes a \textit{degenerate} elliptic operator (see Ref. \cite{Kohn}).
Efforts have been devoted for isolating singularities (degenerate points) 
to show the existence and uniqueness of solution to the boundary value problem of degenerate Poisson equations (see Refs. \cite{Murthy,Chab,Trudinger}).
\footnote{
The key role is played by a non-negative ``control function'' $m\in L^{s}\lr{\Omega}$ such that $g^{jk}\xi_{j}\xi_{k}\geq m\lr{\bol{x}}\abs{\bol{\xi}}^2$ almost everywhere in the domain $\Omega\subset\mathbb{R}^{n}$ and for all $\bol{\xi}\in\mathbb{R}^n$ with $m^{-1}\in L^{t}$ and $s^{-1}+t^{-1}\leq 2/n$. 
So the relevant singularity pertains to the so-called \textit{strict ellipticity} (see \cite{Trudinger}).
} 
In Ref. \cite{Chanillo2}, existence of Green's functions has been studied.
However, the present work is aimed at a different class of $\mathscr{L}$
(in fact, more seriously degenerate operators)
that appear in the theory of diffusion in topologically constrained systems.

In the context of statistical mechanics, $\mathscr{L}$ is the generator of the semigroup representing diffusion processes.
Then, the metric tensor $g^{jk}$ in \eqref{L} is related to the antisymmetric operator $\mc{J}\in\bigwedge^2T\Omega$ 
($\Omega\subset\mathbb{R}^n$ is the phase-space domain)
that generates single particle dynamics by $\dot{\bol{x}}=\mc{J} {dH}$, 
where $H\in C^{\infty}\lr{{\Omega}}$ is the Hamiltonian. 
We have (see Refs. \cite{Sato1,Sato2}): 
\begin{equation}
g^{jk}=-\lr{\mc{J}^2}^{jk}.\label{gjk}
\end{equation} %=\sqrt{\frac{1}{2}\mc{J}^{ij}\mc{J}_{ij}}
When particle motion obeys Hamilton's canonical equations, 
the antisymmetric operator $\mc{J}$ is nothing but the co-symplectic matrix $\mc{J}_{c}$,
and then, $g^{jk}=\delta^{jk}$ is the Euclidean metric
(accordingly, $\mathscr{L}$ is the standard Laplacian).
When some \textit{integrable} topological constraints apply,
particle motion takes the form of a noncanonical Hamiltonian system (see Ref. \cite{Morrison}), 
and the antisymmetric operator $\mc{J}$ defines a degenerate Poisson algebra. If $\bol{x}$ is the coordinate system spanning the invariant measure provided by Liouville's theorem, 
the metric tensor is related to the components $\mc{J}^{ij}$ of the Poisson matrix in such coordinate system (see Refs. \cite{Sato1,Sato2}).
The resulting metric is degenerate, but the nullity (kernel) of the tensor $g^{jk}$ is integrable in terms of Casimir invariants (the Lie-Darboux theorem; see Refs. \cite{DeLeon_3,Arnold_4}). 
The diffusion occurs on the Casimir leaves.

More general class of  topological constraints
may violate the Hamiltonian structure;
this is indeed the subject of our interest.
In order to maintain the energy conservation law, 
$\mc{J}$ must be an antisymmetric matrix, but may have nullity.
\textit{Non-integrable} topological constraints, moreover,
prevent $\mc{J}$ from satisfying the Jacobi identity (see Ref. \cite{Bloch}).
Then, the nullity of $\mc{J}$ does not foliate the phase space.
We call such $\mc{J}$ a \emph{generalized} Poisson matrix. 

In the present work, we consider a generalized Laplace operator (\ref{L})
with a degenerate metric tensor of type  \eqref{gjk},
which we call an \textit{orthogonal Laplacian operator},
and denote it by $\Delta_{\perp}$. 
When the constraint is integrable as a Casimir invariant (i.e., the kernel of $\mc{J}$ foliates the phase space),
$\Delta_{\perp}$ is effectively a Laplacian on the Casimir leaves. 
Then, the diffusion generated by $\Delta_{\perp}$ will flatten the distribution of particles on each Casimir leaf.
However, when the constraint is \emph{non-integrable}, the situation is very different;
we expect (and observe in numerical experiments, see Ref. \cite{Sato2}) that the diffusion generated by $\Delta_{\perp}$
homogenizes the distribution of particles,
or, non-integrable topological constraint cannot sustain inhomogeneity (see Ref. \cite{Sato2}).
To provide this conjecture with mathematical proof,
we show that $-\Delta_{\perp}$ with a non-integrable nullity retains some nice properties of standard elliptic operators, i.e.,
$-\lr{\Delta_{\perp}u,v}$ defines an inner product of a Sobolev-like Hilbert space, and satisfy a Poincar\'e-like inequality $-\lr{\Delta_{\perp}u,u}\geq C\Norm{u}^{2}_{L^2}$.
Applying Riesz's representation theorem, we obtain a unique weak solution of the orthogonal Poisson equation
(that give the stationary distribution of the orthogonal diffusion equation).

We describe the theory for a compact domain in 3-dimensional Euclidean  space $\mathbb{R}^3$
(which is the minimum-dimension space in which the non-integrability can occur;
generalization to compact manifolds of arbitrary dimensions will be mentioned in the concluding remarks).
The action of the antisymmetric matrix $\mc{J}$ can be represented by the cross product with a vector field $\bol{w}$, i.e. 
\begin{equation}
\mc{J} {d f} = \bol{w}\cp\nabla f ,
\end{equation}
% $\left\{f,g\right\}=\nabla f\cdot\bol{w}\cp\nabla g$, 
and then, we may write
\begin{equation}
\Delta_{\perp}
% = -\epsilon^{jlm}\epsilon^{mnk}\partial_j \hat{w}^{l}\hat{w}^{n} \partial_k
=-\nabla\cdot\left[\hb{w}\cp\lr{\hb{w}\cp\nabla}\right],
\end{equation}
where $\hb{w}=\bol{w}/w$. 
We call $\hb{w}$ the dual (or constraining) vector field of the metric. 
In the direction of $\hb{w}$, the orthogonal Laplacian is degenerate.
Physically, it is the direction in which particles cannot move.
In appendix A, we give a short summary of how the operator $\Delta_{\perp}$ arises in the context of constrained diffusion processes.
The non-integrability (or the violation of the Jacobi identity) is detected by the \emph{helicity}
$h=\JI{w}$.
When $h$ is zero, the anti-symmetric bilinear bracket 
$\left\{f,g\right\}=\nabla f\cdot\bol{w}\cp\nabla g$
satisfies the Jacobi identity; 
hence it defines a Poisson manifold (the corresponding dynamics is Hamiltonian). 
The phase space $\Omega$ is foliated by the center (Casimir element) of the 
Poisson algebra, the type of foliation being determined by the Bianchi classification of three dimensional Lie-Poisson algebras (see Ref. \cite{Yoshida}). 
Examples of Hamiltonian systems affected by integrable constraints can be found in Refs. \cite{Yoshida1} and \cite{Morrison}. 
As shown in Ref. \cite{Sato2}, the standard construction of statistical mechanics is then applicable on each Casimir leaf, 
and the stationary solution to the diffusion equation (the boundary value problem for the orthogonal Laplacian) is obtained by means of an H-theorem. 
Notice, however, that such solution is not unique, the multiplicity being determined by the Casimir invariants. 
The same is true for Hamiltonian systems of an arbitrary finite dimension.

When $h\neq 0$ the Jacobi identity is violated and the system ceases to be Hamiltonian.
The phase space is no longer foliated,
because the degenerate direction $\hb{w}$ is not integrable in the sense of the Frobenius theorem (see Ref. \cite{Frankel}). 

The paper is organized as follows. In section II we give formal definitions
of the orthogonal Laplacian operator and the orthogonal Poisson equation.
In section III we introduce a bilinear form 
given by the $L^2$ product of orthogonal components of gradients, 
and show that it satisfies the axioms of inner product on $C_{0}^{1}$ provided that the
constraining vector field has finite helicity.
Then, the bilinear form is combined with the standard $L^2$ inner product to
obtain a Hilbert space $\mc{H}^{\perp}$ as completion of $C^{1}$ with respect to the resulting norm. 
In section IV a trace operator is derived for functions belonging to $\mc{H}^{\perp}$ 
by requiring $\bol{w}$ to be tangent to the boundary.
In section V we obtain a Poincar\'e-like estimate for the orthogonal component
of the gradient of functions in $\mc{H}^{\perp}_{0}$ (the set of functions in $\mc{H}^{\perp}$ with zero trace). 
In particular, we show that
constraining vector fields with finite helicity always guarantee this type of estimate,
and obtain the Poincar\'e constant in terms of the helicty density.
Examples of estimates for specific constraining vector fields are given in section VI.
In section VII we use the Poincar\'e-like estimate
to apply Riesz's representation theorem and prove existence and uniqueness
of solution to the orthogonal Poisson equation in the Hilbert space $\mc{H}^{\perp}_{0}$.
Section VII is for the conclusion.

\section{Orthogonal Laplacian operator and Orthogonal Poisson equation}

%In this paper we examine the \textit{normal Laplacian} $\Delta_{\perp}$ 
%and study existence and uniqueness of solution to the normal Poisson equation,
%to be defined shortly. 
The $3$-dimensional case of $\mathbb{R}^{3}$ is discussed.
Let $\bol{w}\in C^{\infty}\lr{\ov{\Omega}}$ be a smooth non-vanishing vector field in a smoothly bounded connected domain $\Omega\subset\mathbb{R}^{3}$. The direction of $\bol{w}$ is said \emph{parallel}, and the others \emph{orthogonal}, \emph{normal} or \emph{perpendicular}. 

\begin{mydef} (orthogonal and parallel gradients)\label{def8_1}\\
Let $\bol{w}\in C^{\infty}\lr{\ov{\Omega}}$ be a smooth non-vanishing vector field in a smoothly bounded connected domain $\Omega\subset\mathbb{R}^{3}$. In $\Omega$, the orthogonal gradient $\nabla_{\perp}$ and the parallel gradient $\nabla_{\parallel}$ of a real valued function $u\in C^{1}\lr{\Omega}$ with respect to $\bol{w}$ are defined as:
\begin{equation}
\nabla_{\perp}u=\frac{\bol{w}\cp\left(\nabla u\cp\bol{w}\right)}{w^{2}},~~~~\nabla_{\parallel}u=\frac{\bol{w}}{w^{2}}\left(\bol{w}\cdot\nabla u\right).
\end{equation}\label{NGrad}
\end{mydef}

\noindent Notice that $\nabla u=\nabla_{\perp}u+\nabla_{\parallel} u$.

\begin{mydef} (Orthogonal and parallel Laplacian)\label{def8_3}\\ 
The orthogonal Laplacian $\Delta_{\perp}$ and the parallel Laplacian $\Delta_{\parallel}$ of a real valued function $u\in C^{2}\lr{\Omega}$ with orthogonal and parallel gradients given by definition \ref{NGrad} are defined as:
\begin{equation}
\Delta_{\perp}u=\di{\lr{\nabla_{\perp}u}},~~~~\Delta_{\parallel}u=\di{\lr{\nabla_{\parallel}u}}.\label{NLap}
\end{equation}
\end{mydef}

\begin{mydef}(Orthogonal Poisson equation)\label{def8_5}\\
Let $\Omega\subset\mathbb{R}^3$ be a smoothly bounded connected domain with boundary $\p\Omega$.
Let $\phi$ be a known real valued function.
The orthogonal Poisson equation with respect to a real valued function $u$ is a second order partial differential equation:
\begin{equation}
\begin{split}
\Delta_{\perp}u&=\phi~~~~\rm{in}~~~~\Omega,\\
u&=0~~~~\rm{on}~~~~\p\Omega.\label{NLapEq}
\end{split}
\end{equation}
\end{mydef}
\noindent If $\phi\in C\lr{\Omega}$ and $u\in C^{2}\lr{\Omega}$, $u$ is a classical solution to \eqref{NLapEq}.
Evidently \eqref{NLapEq} is not an elliptic PDE (see Refs.\cite{Gilbarg,Evans,Brezis} for the definition of ellipticity) because the coefficient matrix $g^{jk}=-\epsilon^{jlm}\epsilon^{mnk}\hat{w}^{l}\hat{w}^{n}$ is not positive definite (any vector $\xi\in\mathbb{R}^{n}$ aligned with $\bol{w}$ belongs to the kernel of such matrix, $g^{jk}w^{k}=0$ $\forall j$). 

In the following we construct a weak solution to \eqref{NLapEq} under the conditions on the vector field $\bol{w}$ described below.

\begin{remark}
In general, the solution to \eqref{NLapEq} is not unique. For example, if $\bol{w}=\nabla z$, $\phi=0$, and $\Omega$ is periodic in the $x$ and $y$ directions, the orthogonal Laplacian reduces to $\Delta_{\perp}=\p^{2}_{x}+\p_{y}^{2}$, giving solutions of the type $u=u\lr{z}$.
Similarly, if $\hb{w}=\nabla\rho$, with $\rho=\sqrt{x^2+y^2+z^2}$ the radial coordinate of a spherical coordinate system $\lr{\rho,\theta,\phi}$, and $\Omega$ is a sphere of radius $\tilde{\rho}>0$, the orthogonal Laplacian reduces to $\Delta_{\perp}=\frac{1}{\rho^{2}\sin{\theta}}\left[\p_{\theta}\lr{\sin{\theta}\p_{\theta}}+\frac{1}{\sin{\theta}}\p_{\phi}^{2}\right]$, giving solutions of the type $u=u\lr{\rho}$. 
\end{remark}

\section{Orthogonal inner product}

\noindent In what follows, we denote $\hb{w}=\bol{w}/w$.
%Consider the following bilinear product:

\begin{mydef} (Orthogonal gradient product)\\
Let $\bol{w}\in C^{\infty}\lr{\ov{\Omega}}$ be a smooth non-vanishing vector field in a smoothly bounded connected domain $\Omega\subset\mathbb{R}^{3}$. 
The orthogonal gradient product of $u,v\in C^{1}\left(\Omega\right)$ with respect to $\bol{w}$ is the bilinear form:
\begin{equation}
\left(u,v\right)_{\perp}=\lr{\nabla_{\perp}u,\nabla_{\perp}v}
=\int_{\Omega}{\nabla_{\perp}u\cdot\nabla_{\perp}v}\,dV.\label{OBP}
\end{equation} 
where $\lr{f,g}$ denotes the standard $L^{2}$ inner product.
% ($f$ and $g$ may be scalar or vector valued functions).
We define:
\begin{equation}
\Norm{u}^{2}=\lr{u,u},~\Norm{\nabla_{\perp}u}^{2}=\lr{u,u}_{\perp}.
\end{equation}
\end{mydef}

\begin{proposition} (Orthogonal gradient inner product on $C^{1}_{0}\lr{\Omega}$)
\begin{enumerate}
\item
If $h=\bol{w}\cdot\nabla\cp\bol{w}\neq 0$, the bilinear form $\lr{u,v}_{\perp}$ satisfies the axioms of inner-product on $C^{1}_{0}\lr{\Omega}$;
hence, $\Norm{\nabla_{\perp}u}$ is a norm on $C^{1}_{0}\lr{\Omega}$.
\item
The function space $C^{1}_{0}\lr{\Omega}$ equipped with the norm $\Norm{\nabla_{\perp}u}$ is a pre-Hilbert space and, by the theorem of completion, it can be completed to define a Hilbert space $\mc{H}^{\perp}_{0}\lr{\Omega}$. 
\end{enumerate}
\end{proposition}

\begin{proof}
For any $u,v\in C^{1}_{0}\lr{\Omega}$ the bilinear form $\lr{u,v}_{\perp}$ is symmetric, and $\lr{u,u}_{\perp}\geq 0$. 
The only non-trivial statement is $\lr{u,u}_{\perp}=0\iff u=0$. To see this, observe that:
\begin{equation}
\Norm{\nabla_{\perp}u}^2=\int_{\Omega}{\abs{\hb{w}\cp\nabla u}^2}\,dV.
\end{equation}
Due to continuity of the integrand, the integral vanishes if and only if $\hb{w}\cp\nabla u=\bol{0}$ at any point in $\Omega$.
This implies $\nabla u=\alpha\hb{w}$ for some function $\alpha$. If $\alpha\neq 0$ one has $\hb{w}=\alpha^{-1}\nabla u$, contradicting
the non-integrability condition $\hat{h}=w^{-2}h\neq 0$. Hence $\alpha=0$ and $u=\rm{constant}$. From the boundary condition and the continuity of $u$, we conclude $u=0$ in $\Omega$.   
Hence, $\Norm{\nabla_{\perp}u}$ is a norm on $C^{1}_{0}\lr{\Omega}$. 
%We identify with $\mc{H}^{\perp}_{0}\lr{\Omega}$ the completion of $C^{1}_{0}\lr{\Omega}$ with respect to this norm.
\end{proof}

\noindent In order to define a topology, we introduce the following bilinear product:

\begin{mydef} (Orthogonal inner product)\\
Let $\bol{w}\in C^{\infty}\lr{\ov{\Omega}}$ be a smooth non-vanishing vector field in a smoothly bounded connected domain $\Omega\subset\mathbb{R}^{3}$. 
The orthogonal inner product of $u,v\in C^{1}\left(\Omega\right)$ with respect to $\bol{w}$ is the bilinear form:
\begin{equation}
\lr{u,v}_{\mc{H}^{\perp}}=\lr{u,v}+\left(u,v\right)_{\perp}=\lr{u,v}+\lr{\nabla_{\perp}u,\nabla_{\perp}v}
=\int_{\Omega}{\left[uv+\left(\nabla_{\perp}u\cdot\nabla_{\perp}v\right)\right]}\,dV.\label{NIP}
\end{equation} 
%where $\lr{f,g}$ denotes the standard $L^{2}$ inner product ($f$ and $g$ may be scalar or vector valued functions) and $\left(u,v\right)_{\perp}=\lr{\nabla_{\perp}u,\nabla_{\perp}v}$. 
We define:
\begin{equation}
%\Norm{u}^{2}=\lr{u,u},~\Norm{\nabla_{\perp}u}^{2}=\lr{u,u}_{\perp},~
\Norm{u}_{\mc{H}^{\perp}}^{2}=\lr{u,u}_{\mc{H}^{\perp}}={\Norm{u}^{2}+\Norm{\nabla_{\perp}u}^{2}}.
\end{equation}
\end{mydef}

\begin{proposition} (Orthogonal Inner product on $C^{1}\lr{\Omega}$)
\begin{enumerate}
\item
The bilinear form $\lr{u,v}_{\mc{H}^\perp}$ satisfies the axioms of inner-product on $C^{1}\lr{\Omega}$;
hence, $\Norm{u}_{\mc{H}^{\perp}}$ is a norm on $C^{1}\lr{\Omega}$.
\item
The function space $C^{1}\lr{\Omega}$ equipped with the norm $\Norm{u}_{\mc{H}^{\perp}}$ is a pre-Hilbert space and, by the theorem of completion, it can be completed to define a Hilbert space $\mc{H}^{\perp}\lr{\Omega}$. \end{enumerate}
\end{proposition} \label{PropX_2}

\begin{proof}
For any $u,v\in C^{1}\lr{\Omega}$ the bilinear form $\lr{u,v}_{\mc{H}^\perp}$ is symmetric, $\lr{u,u}_{\mc{H}^\perp}\geq 0$, and $\lr{u,u}_{\mc{H}^\perp}=0\iff u=0$. Then, $\Norm{u}_{\mc{H}^{\perp}}$ is a norm on $C^{1}\lr{\Omega}$. We identify with $\mc{H}^{\perp}\lr{\Omega}$ the completion of $C^{1}\lr{\Omega}$ with respect to this norm.
\end{proof}

\begin{remark}
Evidently $H^{1}\lr{\Omega}\subset \mc{H}^{\perp}\lr{\Omega}\subset L^{2}\lr{\Omega}$ ($H^{1}\lr{\Omega}$ is the standard Sobolev space of order 1). 
\end{remark}

\begin{remark}
The derivative $\nabla_{\perp}$ of a function $u\in\mc{H}^{\perp}\lr{\Omega}$ must now be evaluated in the sense of distribution, i.e. for a test function $\phi\in C_{0}^{\infty}\lr{\Omega}$, 
\begin{equation}
\int_{\Omega}{\lr{\nabla_{\perp}u}^{i}\phi}\,dV=
\lim_{n\rightarrow\infty}\epsilon^{ijk}\epsilon^{klm}\int_{\Omega}{\hat{w}^{j}\hat{w}^{m}\frac{\p u_{n}}{\p x^{l}}\phi\,dV}=
-\int_{\Omega}{u\left[ \lr{\nabla_{\perp}\phi}^{i}-\nabla\cdot\lr{\hat{w}^{i}\hb{w}}\phi\right]}\,dV.
\end{equation}
\end{remark}

\section{Trace operator}

\noindent The next task is to study how boundary conditions apply for the members of $\mc{H}^{\perp}\lr{\Omega}$.
Here, we follow the proof given in Ref. \cite{Evans2} for the trace theorem in the standard Sobolev space $W^{1,p}$
and obtain a trace operator for $\mc{H}^{\perp}$. We denote by $\bol{n}$ the outward unit normal vector onto the boundary $\p\Omega$.

\begin{lemma} (Trace operator)\\
Let $\bol{w}\in C^{\infty}\lr{\ov{\Omega}}$ be a smooth non-vanishing vector field in a smoothly bounded connected domain $\Omega\subset\mathbb{R}^{3}$ such that $\bol{n}\cdot\bol{w}=0$ on $\p\Omega$. Then 
there exists a bounded linear operator $T:\mc{H}^{\perp}\lr{\Omega}\rightarrow L^{2}\lr{\p\Omega}$ such that 
\begin{equation}
Tu=u\rvert_{\p\Omega}\,\,\,\,\,\,\,\,if\,\,\,\,u\in \mc{H}^{\perp}\lr{\Omega}\cap C^{1}\lr{\ov{\Omega}}
\end{equation}
and
\begin{equation}
\Norm{Tu}^{2}_{L^{2}\lr{\p\Omega}}\leq C\Norm{u}^{2}_{\mc{H}^{\perp}\lr{\Omega}}\label{trace_bound}
\end{equation}
with a constant $C>0$ independent of $u$.
\end{lemma}
\begin{proof} Take $u\in C^{1}\lr{\ov{\Omega}}$. 
First assume that $\p\Omega$ is flat in some region around the point $\bol{x}_{0}\in\p\Omega$, which lies in the plane $n=0$. Let $B$ and $\hat{B}$ be two concentric balls centered at $\bol{x}_{0}$ of radius $r$ and $r/2$ respectively, with $B\cap\p\Omega$ still in the plane $n=0$. Set $B^{+}=B\cap\ov{\Omega}$. Define $\Gamma=\hat{B}\cap\p\Omega$ and consider a function $\zeta\in C_{c}^{\infty}\lr{B}$ such that $\zeta\geq 0$ in $B$ and $\zeta=1$ in $\hat{B}$. It follows that:
\begin{equation}
\int_{\Gamma}{u^{2}}\,dS\leq \int_{n=0}{\zeta u^{2}}\,dS=-\int_{B^{+}}\frac{\p\lr{\zeta u^{2}}}{\p n}\,dV=
-\int_{B^{+}}{\left[u^{2}\zeta_{n}+2\abs{u}\lr{\sgn u} u_{n}\zeta\right]}\,dV\leq  C\int_{B^{+}}{\lr{u^{2}+u_{n}^{2}}}\,dV.
\end{equation} 
In the last passage Young's inequality $\abs{u}\abs{u_{n}}\leq \lr{\abs{u}^{2}+\abs{u_{n}}^{2}}/2$ was used.
Next, observe that $\bol{n}\cdot\bol{w}=0$ implies $\abs{u_{n}}=\abs{\bol{n}\cdot\nabla u}\leq\abs{\nabla_{\perp}u}$ on $\p\Omega$. Furthermore, we define $\bol{n}$ to be such that $\bol{n}\cdot\bol{w}=0$ holds also in $B^{+}$. Therefore:
\begin{equation}
\int_{\Gamma}{u^{2}}\,dS\leq C\lr{\Norm{u}^{2}+\Norm{\nabla_{\perp}u}^{2}}.
\end{equation} 
If $\p\Omega$ is not flat around $\bol{x}_{0}$, the boundary can be straighten out, and the same procedure applies. 
Since $\p\Omega$ is compact there is a finite number $m$ of domains $\Gamma_{i}$ such that $\cup_{i=1}^{m}\Gamma_{i}=\p\Omega$, leading to result \eqref{trace_bound} for any $u\in C^{1}\lr{\ov{\Omega}}$.

Suppose now that $u\in \mc{H}^{\perp}\lr{\Omega}$. 
We look for a sequence $u_{m}\in C^{1}\lr{\ov{\Omega}}$ approximating $u$ in $\mc{H}^{\perp}\lr{\Omega}$. Since the boundary is smooth, there are a radius $r>0$ and a $C^{1}$ function $\gamma:\mathbb{R}^{2}\rightarrow\mathbb{R}^{3}$ by which $\Omega\cap B\lr{\bol{x}^{0},r}=\left\{\bol{x}\in B\lr{\bol{x}^{0},r}\rvert x^{3}>\gamma\lr{x^{1},x^{2}}\right\}$. The coordinate system $\lr{x^{1},x^{2},x^{3}}$ can be obtained
by relabeling the axes. Define $V=\Omega\cap B\lr{\bol{x}_{0},r/2}$, $\bol{x}^{\epsilon}=\bol{x}+\lambda\epsilon \p_{3}$, and $u_{\epsilon}\lr{\bol{x}}=u\lr{\bol{x}^{\epsilon}}$ $\lr{\bol{x}\in V}$. 
Thanks to the inward shift $\lambda\epsilon \p_{3}$, we have that, $\forall \bol{x}\in V$, $B\lr{\bol{x}^{\epsilon},\epsilon}\subset \Omega\cap B\lr{\bol{x}_{0},r}$ for a sufficiently large $\lambda>0$ and small $\epsilon>0$.  
By the density argument, we can take a Cauchy sequence $v^{\epsilon}\in C^{1}\lr{\ov{V}}$ $\lr{\bol{x}\in V}$ approximating $u^{\epsilon}$ and converging to $u$ in $\mc{H}^{\perp}\lr{V}$ for $\epsilon\rightarrow 0$.
On the other hand, since $\p\Omega$ is compact, there are a finite number $N$ of points $\bol{x}_{0i}$ such that
$\p\Omega\subset\cup_{i=1}^{N}B^{0}\lr{\bol{x}_{0i},r_{i}/2}$, where $B^{0}$ is the open ball, and functions $v_{i}\in C^{1}\lr{\ov{V}^{i}}$, $1\leq i\leq N$, $V^{i}=\Omega\cap B^{0}\lr{\bol{x}_{0i},r_{i}/2}$, such that $\Norm{v_{i}-u}_{\mc{H}^{\perp}\lr{V^{i}}}\leq\delta$ with $\delta >0$. 
Next, choose a domain $V^{0}\subset\subset \Omega$ such that $\Omega\subset\cup_{i=0}^{N}V^{i}$,
$\Norm{v_{0}-u}_{\mc{H}^{\perp}\lr{V^{0}}}\leq\delta$, $v_{0}\in C^{1}\lr{\ov{V}^{0}}$. Let $\left\{\zeta_{i}\right\}_{i=0}^{N}$ be the partition to unity corresponding to the open sets $\left\{V^{0},B^{0}\lr{\bol{x}_{1}^{0},r_{1}/2},...\right\}$. By defining $v=\sum_{i=0}^{N}\zeta_{i}v_{i}\in C^{1}\lr{\ov{\Omega}}$ and noting that $u=\sum_{i=0}^{N}\zeta_{i}u$, we have
$\Norm{\nabla_{\perp}\lr{v-u}}_{L^{2}\lr{\Omega}}\leq\sum_{i=0}^{N}\Norm{\nabla_{\perp}\lr{\zeta_{i}v_{i}-\zeta_{i}u}}_{L^{2}\lr{V^{i}}}\leq C\sum_{i=0}^{N}\Norm{v_{i}-u}_{\mc{H}^{\perp}\lr{V^{i}}}\leq C\delta \lr{N+1}$, which goes to zero when $\delta\rightarrow 0$. Therefore, $v$ can be taken as an element of the sequence $u_{m}$.
Then, the trace sequence $Tu_{m}$ is a Cauchy sequence in $L^{2}\lr{\p\Omega}$ satisfying $\Norm{Tu_{m}-Tu_{n}}_{L^{2}\lr{\p\Omega}}^{2}\leq C\lr{\Norm{u_{m}-u_{n}}^{2}+\Norm{\nabla_{\perp}\lr{u_{m}-u_{n}}}^{2}}$. We can thus define the $L^{2}\lr{\p\Omega}$ limit $Tu=\lim_{m\rightarrow\infty}Tu_{m}$. 
\end{proof}

\noindent Define $\mc{H}^{\perp}_{0}={\rm{ker}}\lr{T}$ to be the set of functions in $\mc{H}^{\perp}\lr{\Omega}$ that have zero trace $Tu=0$.

%\newpage

\section{Poincar\'e-like inequality}

We introduce an extended bounded domain $\tilde{\Omega} \supset \Omega$ such that $\partial\tilde{\Omega}\cap\partial\Omega=\emptyset$. 

\begin{lemma}\label{LM1}
Let $\bol{w}\in C^{\infty}(\tilde{\Omega})$ be a smooth vector field in a bounded domain $\tilde{\Omega}\subset\mathbb{R}^3$
%. Let $\Omega\subset\tilde{\Omega}$ be a smoothly bounded connected domain such that $\p\tilde{\Omega}\cap\p\Omega=\emptyset$ and 
satisfying $\bol{n}\cdot\bol{w}=0$ on $\p\Omega$. 
Assume $\abs{h}=\abs{\JI{w}}\geq \delta>0$ in $\tilde{\Omega}$. Then there exists a finite open cover $\left\{U_{1},...,U_{\alpha}\right\}$ of $\ov{\Omega}$, and coordinates $\lr{\ell_{i},\psi_{i},\theta_{i}}\in C^{\infty}\lr{\ov{U}_{i}}$, $i=1,...,\alpha$,  such that $\cup_{i=1}^{\alpha}U_{i}\supset\ov{\Omega}$ and $\nabla\cp\hb{w}=\nabla\psi_{i}\cp\nabla\theta_{i}$ in $\ov{U}_{i}$. 
\end{lemma}

\begin{proof}
%The case $h>0$ is considered. 
Observe that $\epsilon=\inf_{\ov{\Omega}}\abs{h}\geq\delta$.
Since $\bol{w}$ is smooth over the closed bounded interval $\ov{\Omega}$, in such interval $w=\abs{\bol{w}}\leq M$ and $\abs{\nabla\cp\bol{w}}\leq N$ for some positive real constants $M$ and $N$. If $\phi$ is the angle between $\bol{w}$ and $\nabla\cp\bol{w}$, one has $\abs{h}=w\abs{\nabla\cp\bol{w}}\abs{\cos{\phi}}\geq\epsilon$. Therefore $w\abs{\nabla\cp\bol{w}}\geq \epsilon$, $w\geq\epsilon/N>0$, and $\abs{\nabla\cp\bol{w}}\geq \epsilon/M>0$ in $\ov{\Omega}$. 
Next observe that, since $w\neq 0$, the normalized vector field $\hb{w}=\bol{w}/w$ is well defined in $\ov{\Omega}$. In addition, the helicity $\hat{h}=\JI{\hb{w}}$ of the normalized vector field $\hb{w}$ satisfies $\lvert\hat{h}\rvert=w^{-2}\abs{h}\geq M^{-2}~\epsilon>0$ in $\ov{\Omega}$. 

Define the 1-form $\omega=\hat{w}^{i} dx^{i}$. Then, the 2-form $\eta=d\omega=\lr{\nabla\cp\hb{w}}^{i} \ast dx^{i}$ is closed with a constant rank $2$, because $h\neq 0$ implies $\nabla\cp\hb{w}\neq \bol{0}$ in $\tilde{\Omega}$. Then, Darboux's theorem (see Refs. \cite{DeLeon_3,Arnold_4}) guarantees that, for every $\bol{x}_{i}\in\tilde{\Omega}$, there exists a neighborhood $U_{i}\subset\tilde{\Omega}$ of $\bol{x}_{i}$ such that
\begin{equation}
\eta=d\psi_{i}\w d\theta_{i}~~{\rm{in}}~~\ov{U}_{i},
\end{equation} 
where $\psi_{i}\in C^{\infty}\lr{\ov{U}_{i}}$ and $\theta_{i}\in C^{\infty}\lr{\ov{U}_{i}}$ span local coordinates in $U_{i}$. Hence $\nabla\cp\hb{w}=\nabla\psi_{i}\cp\nabla\theta_{i}$ in $\ov{U}_{i}$. 
%Without loss of generality, $U_{i}$ is chosen to be a closed set of finite measure.
Furthermore $\cup_{i}U_{i}\supset\ov{\Omega}$. Since $\ov{\Omega}$ is compact, there exists a finite subcover $\left\{U_{i}\right\}=\left\{U_{1},...,U_{\alpha}\right\}$ of $\ov{\Omega}$ such that $\cup_{i=1}^{\alpha}U_{i}\supset\ov{\Omega}$ for some $\alpha\in\mathbb{N}$. 
The intersection of the level sets $\psi_{i}=\rm{constant}$ and $\theta_{i}=\rm{constant}$ defines a curve $\Gamma_{i}\subset \ov{U}_{i}$ whose tangent vector is
\begin{equation}
\p_{\ell_{i}}=\frac{\p\bol{x}}{\p\ell_{i}}=\frac{\nabla\psi_{i}\cp\nabla\theta_{i}}{\abs{\nabla\psi_{i}\cp\nabla\theta_{i}}}=\frac{\nabla\cp\hb{w}}{\abs{\nabla\cp\hb{w}}},
\end{equation}
where $\ell_{i}$ is a parameter measuring the length of the curve $\Gamma_{i}$.
The parameter $\ell_{i}$ can be taken as a local smooth coordinate in $\ov{U}_{i}$.
The Jacobian of the coordinate change $\lr{x,y,z}\mapsto\lr{\ell_{i},\psi_{i},\theta_{i}}$ is
\begin{equation}
\nabla{\ell_{i}}\cdot\nabla\psi_{i}\cp\nabla\theta_{i}=\lr{\nabla\ell_{i}\cdot\p_{\ell_{i}}}\abs{\nabla\psi_{i}\cp\nabla\theta_{i}}=\abs{\nabla\cp\hb{w}}\geq \lvert\hat{h}\rvert>0.
\end{equation} 
This completes the proof of lemma.
\end{proof}

\begin{lemma}\label{LM2}
Let $\bol{w}\in C^{\infty}(\tilde{\Omega})$ be a smooth vector field in a bounded domain $\tilde{\Omega}\subset\mathbb{R}^3$ satisfying $\bol{n}\cdot\bol{w}=0$ on $\p\Omega$. 
Assume $\abs{h}=\abs{\JI{w}}\geq \delta>0$ in $\tilde{\Omega}$. Then there exists a vector field $\bol{a}_{\perp}\in C^{\infty}\lr{\ov{\Omega}}$ with the following properties in $\Omega$:
\begin{subequations}\label{aperp}
\begin{align}
\bol{a}_{\perp}\cdot\bol{w}&=0,\\
\nabla\cdot\bol{a}_{\perp}=\abs{\nabla\cp\hb{w}}&\geq M^{-2}\epsilon>0,\\
\abs{\bol{a}_{\perp}}&\leq \nu,
\end{align}
\end{subequations}
where $\epsilon=\inf_{\ov{\Omega}}{\abs{h}}$, $M=\sup_{\ov{\Omega}}\abs{\bol{w}}$, and $\nu=\sup_{\ov{\Omega}}\abs{\bol{a}_{\perp}}$.
\end{lemma}

\begin{proof}
Let $\left\{U_{i}\right\}$ and $\lr{\ell_{i},\psi_{i},\theta_{i}}$, $i=1,...,\alpha$, 
be the finite open cover of $\ov{\Omega}$ and the local coordinate systems obtained in lemma \ref{LM1}. 
We denote by $\left\{V_{i}\right\}=\left\{V_{1},...,V_{\alpha}\right\}$ the adjustment of $\left\{U_{i}\right\}$ with open sets $V_{i}\subseteq U_{i}$ such that $\cup_{i=1}^{\alpha}V_{i}=\ov{\Omega}$, the intersections $V_{i}\cap V_{j}$ are either empty or of finite measure, and there is no set $V_{i}$ such that $V_{i}\subseteq V_{j}$ for some $V_{j}$.
This adjustment is always possible since the intersection of a finite number of open sets in a metric space is always open, and the open sets $U_{i}$ have non empty intersections. The vector field
\begin{equation}
\bol{a}_{\perp 1}=\hb{w}\cp\nabla\ell_{1}~~~~\rm{in}~~~~\ov{V}_{1}
\end{equation}
has the following properties in $\ov{V}_{1}$:
\begin{subequations}
\begin{align}
\bol{a}_{\perp 1}\cdot\bol{w}&=0,\\
\nabla\cdot\bol{a}_{\perp 1}=\nabla\ell_{1}\cdot\nabla\cp\hb{w}=\abs{\nabla\cp\hb{w}}
&\geq \lvert\hat{h}\rvert\geq M^{-2}\epsilon>0,\\
\abs{\bol{a}_{\perp 1}}&\leq\abs{\nabla\ell_{1}}\leq\nu_{1}.
\end{align}
\end{subequations}
To derive the last equation, we used the fact that $\ell_{1}$ is smooth in the closed interval $\ov{V}_{1}$ 
and, therefore, $\nabla\ell_{1}$ is bounded in $\ov{V}_{1}$ by a positive real constant $\nu_{1}=\sup_{\ov{V}_{1}}\abs{\nabla\ell_{1}}$.

The next step is to extend the function $\ell_{1}$ to the whole $\ov{\Omega}$.
By construction there exists a neighborhood $V_{2}$ that has a finite measure intersection with $V_{1}$. Let $\lr{\ell_{2},\psi_{2},\theta_{2}}$ be the local set of coordinates in $V_{2}$. 
In $\ov{V}_{1}\cap \ov{V}_{2}$, $\nabla\ell_{1}\cdot\nabla\cp\hb{w}=\nabla\ell_{2}\cdot\nabla\cp\hb{w}=\abs{\nabla\cp\hb{w}}$ and $\nabla\cp\hb{w}=\nabla\psi_{1}\cp\nabla\theta_{1}=\nabla\psi_{2}\cp\nabla\theta_{2}$. Hence, $\nabla\lr{\ell_{1}-\ell_{2}}\cdot\nabla\psi_{1}\cp\nabla\theta_{1}=\nabla\lr{\ell_{1}-\ell_{2}}\cdot\nabla\psi_{2}\cp\nabla\theta_{2}=0$. The coordinates $\ell_{1}$ and $\ell_{2}$ differ up to a function $\sigma_{12}\lr{\psi_{1},\theta_{1}}=\sigma_{12}\lr{\psi_{2},\theta_{2}}$ of $\psi_{1}$ and $\theta_{1}$ (or $\psi_{2}$ and $\theta_{2}$), i.e.,
\begin{equation}
\ell_{1}-\ell_{2}=\sigma_{12}\lr{\psi_{1},\theta_{1}}=\sigma_{12}\lr{\psi_{2},\theta_{2}}.
\end{equation}
Since $\ell_{1}$ and $\ell_{2}$ are smooth, the function $\sigma_{12}$ is also smooth. 
%Since $\sigma_{12}$ is smooth on the closed interval $\ov{V}_{1}\cap\ov{V}_{2}$, %it is also Lipschitz continuous. 
%The Lipschitz constant is given by $\nu_{1}+\nu_{2}$ because $\abs{\nabla\sigma_{12}}\leq\abs{\nabla\ell_{1}}+\abs{\nabla\ell_{2}}\leq\nu_{1}+\nu_{2}$. Here $\nu_{i}=\sup_{V_{i}}\abs{\nabla\ell_{i}}$. 
By the Whitney extension theorem (see Refs. \cite{Shane, Whitney}), $\sigma_{12}$ can be extended to the whole $\ov{V}_{2}$ as a smooth function. Furthermore, since $\sigma_{12}$ does not depend on $\ell_{2}$, such extension can be performed
in the 2-dimensional space of $\lr{\psi_{2},\theta_{2}}$ so that the extended function $\sigma^{\ast}_{12}\lr{\psi_{2},\theta_{2}}$ can be made independent of $\ell_{2}$. %Note that the extended function $\sigma^{\ast}_{12}$ will be compatible with all values of $\ell_{2}$ provided that the
%extension is performed from the set ${\rm P}_{\ell_{2}}\lr{\ov{V}_{1}\cap\ov{V}_{2}}=\left\{\lr{\psi_{2},\theta_{2}}\in\ov{V}_{1}\cap\ov{V}_{2}\right\}$ given by the superpositions of the projections of the level sets $\ell_{2}={\rm constant}$ in $\lr{\psi_{2},\theta_{2}}$ space. 

Next, we define
\begin{subequations}
\begin{align}
\mc{L}_{2}&=\ell_{1}~~~~\rm{in}~~~~\ov{V}_{1}-\ov{V}_{1}\cap \ov{V}_{2},\\
\mc{L}_{2}&=\ell_{2}+\sigma^{\ast}_{12}~~~~\rm{in}~~~~\ov{V}_{2}.
\end{align}
\end{subequations}
By the construction, $\mc{L}_{2}$ is smooth in $\ov{V}_{1}\cup \ov{V}_{2}$.
Consider the vector field 
\begin{equation}
\bol{a}_{\perp 2}=\hb{w}\cp\nabla\mc{L}_{2}~~~~\rm{in}~~~~\ov{V}_{1}\cup \ov{V}_{2}.
\end{equation} 
$\bol{a}_{\perp 2}$ has the following properties in $\ov{V}_{1}\cup \ov{V}_{2}$:
\begin{subequations}\label{a2}
\begin{align}
\bol{a}_{\perp 2}\cdot\bol{w}&=0,\\
\nabla\cdot\bol{a}_{\perp 2}=\abs{\nabla\cp\hb{w}}%\abs{\nabla\ell\cdot\nabla\cp\hb{w}}=\abs{\nabla\cp\hb{w}}
&\geq \lvert\hat{h}\rvert\geq M^{-2}\epsilon>0,\\
\abs{\bol{a}_{\perp 2}}&\leq\abs{\nabla\mc{L}_{2}}\leq\nu_{2}.
\end{align}
\end{subequations}
In the last equation we used the fact that $\mc{L}_{2}$ is smooth in the closed interval $\ov{V}_{1}\cup\ov{V}_{2}$ 
and therefore $\nabla\mc{L}_{2}$ is bounded in $\ov{V}_{1}\cup\ov{V}_{2}$ by a positive real constant $\nu_{2}=\sup_{\ov{V}_{1}\cup\ov{V}_{2}}\abs{\nabla\mc{L}_{2}}$.

The procedure above can be repeated to further extend the domain of the function $\mc{L}_{2}$ by adding one by one the neighboring domains ${V_{3},...,V_{\alpha}}$. 
Notice that, at each step $i$, the difference $\sigma_{i,i+1}=\mc{L}_{i}-\ell_{i+1}$ will be a function of $\lr{\psi_{i+1},\theta_{i+1}}$. 
%When adding a new neighborhood, say $V_{n}$, one must take care in extending the function $\sigma=\sigma\lr{\psi_{n},\theta_{n}}=\mc{L}-\ell_{n}$ so that it is compatible with its values in all regions that intersect with the domain of $\mc{L}$.
Finally, we obtain a smooth function $\mc{L}_{\alpha}$ defined in the whole $\ov{\Omega}$.
The vector field ${\bol{a}}_{\perp\alpha}=\hb{w}\cp\nabla{\mc{L}_{\alpha}}$ inherits the properties \eqref{a2} in the whole $\ov\Omega$:
\begin{subequations}\label{aalpha}
\begin{align}
\bol{a}_{\perp \alpha}\cdot\bol{w}&=0,\\
\nabla\cdot\bol{a}_{\perp \alpha}=\abs{\nabla\cp\hb{w}}%\abs{\nabla\ell\cdot\nabla\cp\hb{w}}=\abs{\nabla\cp\hb{w}}
&\geq \lvert\hat{h}\rvert\geq M^{-2}\epsilon>0,\\
\abs{\bol{a}_{\perp \alpha}}&\leq\abs{\nabla\mc{L}_{\alpha}}\leq\nu.
\end{align}
\end{subequations}
%Note that $\sum_{i=0}^{\alpha}\nu_{i}>0$ since $\nabla\cdot\bol{a}_{\perp\alpha}$ is finite and therefore $\bol{a}_{\perp}$ is different from zero almost everywhere in $\ov{\Omega}$.
The proof is completed by setting $\bol{a}_{\perp}=\bol{a}_{\perp\alpha}$.
\end{proof}

\begin{theorem}\label{PI}
Let $\bol{w}\in C^{\infty}(\tilde{\Omega})$ be a smooth vector field in a bounded domain $\tilde{\Omega}\subset\mathbb{R}^3$ satisfying $\bol{n}\cdot\bol{w}=0$ on $\p\Omega$. 
Assume $\abs{h}=\abs{\JI{w}}\geq \delta>0$ in $\tilde{\Omega}$. Then the following estimate holds:
\begin{equation}
\Norm{\nabla_{\perp} u}\geq \frac{\epsilon}{2M^2\nu}\Norm{u}\,\,\,\,\,\,\,\,\,\forall u\in {\mc{H}}^{\perp}_{0}\lr{\Omega},
\end{equation}
where $\epsilon=\inf_{\ov{\Omega}}{\abs{h}}$, $M=\sup_{\ov{\Omega}}\abs{\bol{w}}$, and $\nu=\sup_{\ov{\Omega}}\abs{\bol{a}_{\perp}}$ with $\bol{a}_{\perp}$ given in lemma \ref{LM2}.
\end{theorem}

\begin{proof}
For $u\in\mc{H}^{\perp}_{0}\lr{\Omega}$ let us evaluate the integral
\begin{equation}
\mc{I}=\int_{\Omega}{\nabla\cdot\lr{u^2\bol{a}_{\perp}}}\,dV,
\end{equation}
where $\bol{a}_{\perp}$ is given by \eqref{aalpha}.
Since $u$ is not differentiable, the divergence appearing in the integral must be evaluated in the distribution sense. In analogy to the definition given in Ref. \cite{Pigola}, the \textit{distributional divergence} $\nabla\cdot\mf{A}_{\perp}$ of a vector field $\mf{A}_{\perp}\in L^{2}\lr{\Omega}$ such that $\mf{A}_{\perp}\cdot\bol{w}=0$ in $\Omega$ is defined by
\begin{equation}
\lr{\psi, \nabla\cdot\mf{A}_{\perp}}=-\int_{\Omega}{\mf{A}_{\perp}\cdot\nabla_{\perp}\psi~}\,dV,~~~~\forall\psi\in \mc{H}_{0}^{\perp}\lr{\Omega}.
\end{equation}
By the definition,
\begin{dmath}
\lr{\psi,\nabla\cdot\lr{u^2\bol{a}_{\perp}}}=-\int_{\Omega}{u^2\bol{a}_{\perp}\cdot\nabla_{\perp}\psi}\,dV=
-\int_{\Omega}{\bol{a}_{\perp}\cdot\left[{\nabla_{\perp}\lr{u^2\psi}-2\,\psi\, u\,\nabla_{\perp}u}\right]}\,dV=
\lr{u^2\psi,\nabla\cdot\bol{a}_{\perp}}+2\int_{\Omega}{\psi\,u\,\nabla_{\perp}u\cdot\bol{a}_{\perp}}\,dV=
\lr{\psi,u^{2}\nabla\cdot\bol{a}_{\perp}+2\,u\,\nabla_{\perp}u\cdot\bol{a}_{\perp}}.\label{eq3_3}
\end{dmath}
Hence, $\nabla\cdot\lr{u^2\bol{a}_{\perp}}=u^2\nabla\cdot\bol{a}_{\perp}+2\,u\,\nabla_{\perp}u\cdot\bol{a}_{\perp}$ in the sense of distributions. Similarly, we may calculate $\nabla\cdot\lr{u\,\bol{a}_{\perp}}=u\,\nabla\cdot\bol{a}_{\perp}+\nabla_{\perp}u\cdot\bol{a}_{\perp}$. 
The next step is to show that $\mc{I}=0$. To see this observe that
\begin{equation}
\lr{\psi,\nabla\cdot\lr{u\,\bol{a}_{\perp}}}+\int_{\Omega}{u\,\bol{a}_{\perp}\cdot\nabla_{\perp}\psi}\,dV=0.
\end{equation}
For $\psi=u$, this evaluates as
\begin{equation}
\int_{\Omega}{\left[u\,\nabla\cdot\lr{u\,\bol{a}_{\perp}}+u\,\nabla_{\perp}u\,\cdot\bol{a}_{\perp}\right]}\,dV=\int_{\Omega}{\left[u^2\nabla\cdot\bol{a}_{\perp}+2\,u\,\nabla_{\perp}u\cdot\bol{a}_{\perp}\right]}\,dV=0.\label{eq3_5}
\end{equation}
Recalling equation \eqref{eq3_3}, we obtain $\mc{I}=0$. 
Finally, using this result with $\abs{\bol{a}_{\perp}}\leq \nu$ and $\nabla\cdot\bol{a}_{\perp}\geq M^{-2}\epsilon$, we obtain
\begin{equation}
2\Norm{u}\Norm{\nabla_{\perp}u}\nu\geq 2\abs{\int_{\Omega}{u\,\nabla_{\perp}u\cdot\bol{a}_{\perp}}\,dV}=
\abs{\int_{\Omega}u^2\nabla\cdot\bol{a}_{\perp}\,dV}\geq M^{-2}\epsilon\Norm{u}^2,
\end{equation}
which proves the theorem.
\end{proof}

\begin{corollary} \label{cor1}
Let $\bol{w}\in C^{\infty}\lr{\ov{\Omega}}$ be a smooth non-vanishing vector field in a smoothly bounded connected domain 
$\Omega\subset\mathbb{R}^{3}$ such that $\bol{n}\cdot\bol{w}=0$ on $\p\Omega$. Let $\bol{a}_{\perp}\in C^{1}\lr{\Omega}$ be a vector field with the following properties in $\Omega$:
\begin{subequations}
\begin{align}
\bol{a}_{\perp}\cdot\bol{w}&=0,\\
\inf_{\Omega}\abs{\nabla\cdot\bol{a}_{\perp}}&=\epsilon>0,\\
\sup_{\Omega}\abs{\bol{a}_{\perp}}&=\nu,
\end{align}
\end{subequations}
Then
\begin{equation}%\label{cor1}
\Norm{\nabla_{\perp} u}\geq \frac{\epsilon}{2\nu}\Norm{u}\,\,\,\,\,\,\,\,\,\forall u\in {\mc{H}}^{\perp}_{0}\lr{\Omega}.
\end{equation}  
\end{corollary}

\begin{proof}
This result follows from the proof of theorem \ref{PI}. 
\end{proof}

\noindent The following statement clarifies the geometrical meaning carried by the vector field $\bol{a}_{\perp}$.

\begin{proposition}\label{prop3.2}
Let $\bol{w}\in C^{\infty}\lr{\ov{\Omega}}$ be a smooth non-vanishing vector field in a smoothly bounded connected domain $\Omega\subset\mathbb{R}^{3}$ such that $\bol{n}\cdot\bol{w}=0$ on $\p\Omega$.
Assume that there exists a vector field $\bol{a}_{\perp}\in C^{1}\lr{\Omega}$ such that $\bol{a}_{\perp}\cdot\bol{w}=0$ and $\abs{\nabla\cdot\bol{a}_{\perp}}\geq\epsilon> 0$ in $\Omega$.
Then, there is no integral surface $\p U$ enclosing a finite volume $U\subset\Omega$ with $\hb{w}$ as normal vector.  
\end{proposition}

\begin{proof}
Suppose that such a surface exists. Observe that $\p U$ is smooth because its normal $\hb{w}$ is, by hypothesis, smooth. Then the divergence theorem holds,
\begin{equation}
0=\int_{\p U}{\bol{a}_{\perp}\cdot\hb{w}}\,dS=\int_{U}{\nabla\cdot\bol{a}_{\perp}}\,dV,
\end{equation}
which contradicts the assumption $\abs{\nabla\cdot\bol{a}_{\perp}}\geq\epsilon> 0$. 
\end{proof}

\section{Examples}

\begin{example}
Consider the vector field $\bol{w}=\hb{w}=\nabla x$ in a smoothly bounded connected domain $\Omega\subset\mathbb{R}^{3}$ such that $x\in\left[0,x_{\ast}\right]$, $y\in\left[0,y_{\ast}\right]$, and $z\in\left[0,z_{\ast}\right]$. Suppose that the domain is periodic in the $x$ direction with period $x_{\ast}$. Evidently $h=\hat{h}=0$.
However, setting $\bol{a}_{\perp}=z\nabla x\cp\nabla y$, one has $\abs{\bol{a}_{\perp}}=z$ (implying $\abs{\bol{a}_{\perp}}\leq z_{\ast}$) and $\nabla\cdot\bol{a}_{\perp}=1$. Hence, from corollary \ref{cor1} and taking $u=0$ at $y=0$, $y=y_{\ast}$, $z=0$, and $z=z_{\ast}$ (this is possible since on these planes $\bol{w}\cdot\bol{n}=0$ and the trace operator can be defined),
\begin{equation}
\Norm{\nabla_{\perp}u}\geq~z^{-1}_{\ast}\Norm{u}.
\end{equation}
Similarly, for vector fields in the form $\bol{w}=\alpha\nabla x+\beta\nabla y$, the Poincar\'e-like inequality can be obtained by setting $\bol{a}_{\perp}=z\nabla x\cp\nabla y$ (provided that $\bol{w}\cdot\bol{n}=0$ on $\p\Omega$). This example shows that the non-vanishing of the helicity $h$ is not a necessary condition for the Poincar\'e-like inequality to hold.
\end{example}

\begin{example}(Poincar\'e-like inequality for finite helicity epi-2D vector fields)\label{ex_2}\\
Let $\bol{w}\in C^{\infty}\lr{\ov{\Omega}}$ be a smooth vector field in a smoothly bounded connected domain $\Omega\subset\mathbb{R}^{3}$ such that $\bol{n}\cdot\bol{w}=0$ on $\p\Omega$. 
Assume $\abs{h}=\abs{\JI{w}}\geq \epsilon>0$ in $\Omega$ and that $\bol{w}$ admits the
Clebsch parametrization\footnote{See Ref.\cite{YClebsch} on the Clebsch parametrization of vector fields in $\mathbb{R}^3$. A vector field admitting the decomposition of equation \eqref{epi2d} is called and epi-2D vector field (see Ref.\cite{YEpi2D}).}
\begin{equation}
\bol{w}=\nabla\phi+\psi\,\nabla\theta,\label{epi2d}
\end{equation}
where $\phi$, $\psi$, and $\theta$ are smooth functions in $\ov{\Omega}$.
Then:
\begin{equation}\label{lemma3}
\Norm{\nabla_{\perp} u}\geq \frac{\epsilon}{2\nu}\Norm{u}\,\,\,\,\,\,\,\,\,\forall u\in {\mc{H}}^{\perp}_{0}\lr{\Omega},
\end{equation}  
%with $c_{\perp}$ a positive real constant depending only on $\Omega$ and $\bol{w}$.
with $\nu=\sup_{\ov{\Omega}}w\abs{\nabla\phi}$.
\end{example}

\begin{proof}
Consider the vector field
\begin{equation}
\bol{a}_{\perp}=\psi\,\nabla\theta\cp\nabla\phi~~~~\rm{in}~~~~\Omega.
\end{equation}
$\bol{a}_{\perp}$ has the following properties in $\Omega$:
\begin{subequations}
\begin{align}
\bol{a}_{\perp}\cdot\bol{w}&=0,\\
\abs{\nabla\cdot\bol{a}_{\perp}}=\abs{\nabla\psi\cdot\nabla\theta\cp\nabla\phi}=\abs{\JI{w}}=\abs{h}&\geq\epsilon> 0,\\
\abs{\bol{a}_{\perp}}&\leq w\abs{\nabla\phi}\leq\nu,
\end{align}
\end{subequations}
where in the last equation we used the fact that $\bol{w}$ and $\phi$ are smooth in $\ov{\Omega}$ and therefore
$w\abs{\nabla\phi}$ is bounded in $\Omega$ by some positive real constant $\nu$. 
The proof is completed by applying corollary \ref{cor1}. % and by setting $c_{\perp}=1/\lr{2\nu}$.
\end{proof}

\begin{example}(Poincar\'e-like inequality for finite helicity and finite divergence vector fields)\label{ex_3}\\
Let $\bol{w}\in C^{\infty}\lr{\ov{\Omega}}$ be a smooth vector field in a smoothly bounded connected domain $\Omega\subset\mathbb{R}^{3}$ such that $\bol{n}\cdot\bol{w}=0$ on $\p\Omega$. 
Assume $\abs{h}=\abs{\JI{w}}\geq \epsilon>0$ and $\abs{\nabla\cdot\hb{w}}\geq \delta>0$ in $\Omega$. Then:
%Assume $\lvert{\hat{h}\rvert}\geq \epsilon>0$ in $\Omega$, where $\hat{h}=\hb{w}\cdot\nabla\cp\hb{w}$ and $\hb{w}=\bol{w}/w$. Then:
\begin{equation}\label{lemma4}
\Norm{\nabla_{\perp} u}\geq \frac{\delta\epsilon}{2M^2\nu}\Norm{u}\,\,\,\,\,\,\,\,\,\forall u\in {\mc{H}}^{\perp}_{0}\lr{\Omega},
\end{equation}  
%with $c_{\perp}$ a positive real constant depending only on $\Omega$ and $\bol{w}$.
with $M=\sup_{\ov{\Omega}}$w and $\nu=\sup_{\ov{\Omega}}\lvert\hb{w}\cp\lr{\nabla\cp\hb{w}}-\nabla\log{\hat{h}}\rvert$.
\end{example}

\begin{proof}
%The case $h>0$ is considered.
Since $\bol{w}$ is smooth over the closed bounded interval $\ov{\Omega}$, in such interval $w=\abs{\bol{w}}\leq M$ and $\abs{\nabla\cp\bol{w}}\leq N$ for some positive real constants $M$ and $N$. If $\theta$ is the angle between $\bol{w}$ and $\nabla\cp\bol{w}$, one has $\abs{h}=w\abs{\nabla\cp\bol{w}}\abs{\cos{\theta}}\geq\epsilon$. Therefore $w\abs{\nabla\cp\bol{w}}\geq \epsilon$ and, from the boundedness of $\nabla\cp\bol{w}$, it follows that $w\geq\epsilon/N>0$. Similarly, $\abs{\nabla\cp\bol{w}}\geq \epsilon/M>0$. 

Next observe that, since $w\neq 0$, the normalized vector field $\hb{w}=\bol{w}/w$ is well defined in $\ov{\Omega}$. In addition, the helicity $\hat{h}=\JI{\hb{w}}$ of the normalized vector field $\hb{w}$ satisfies $\lvert\hat{h}\rvert=w^{-2}\abs{h}\geq M^{-2}~\epsilon>0$.

Consider the vector field
\begin{equation}
\bol{a}_{\perp}=\frac{1}{\hat{h}}\lr{\hb{b}-\nabla\log{\hat{h}}}\cp\hb{w},
\end{equation}
where $\hb{b}=\hb{w}\cp\lr{\nabla\cp\hb{w}}$ is the \textit{field force} of $\hb{w}$. Next, observe that $\nabla\cp\hb{w}=\hb{b}\cp\hb{w}+\hat{h}\hb{w}$ and therefore
\begin{dmath}
\nabla\cdot\bol{a}_{\perp}
=\nabla\cdot\left[\frac{\hb{b}\cp\hb{w}}{\hat{h}}+\nabla\lr{\frac{1}{\hat{h}}}\cp\hb{w}\right]=
\nabla\cdot\left[\frac{\nabla\cp\hb{w}-\hat{h}\hb{w}}{\hat{h}}+\nabla\lr{\frac{1}{\hat{h}}}\cp\hb{w}\right]=
-\nabla\cdot\hb{w}.
\end{dmath}
Hence, $\bol{a}_{\perp}$ has the following properties in $\Omega$:
\begin{subequations}
\begin{align}
\bol{a}_{\perp}\cdot\bol{w}&=0,\\
\abs{\nabla\cdot\bol{a}_{\perp}}=\abs{\nabla\cdot\hb{w}}&\geq\delta> 0,\\
\abs{\bol{a}_{\perp}}&\leq\abs{\frac{\hb{b}-\nabla\log{\hat{h}}}{\hat{h}}}\leq M^2\epsilon^{-1}\abs{\hb{b}-\nabla\log{\hat{h}}}\leq M^2\epsilon^{-1}\nu,
\end{align}
\end{subequations}
where in the last equation we used the fact that $\bol{w}$ is smooth in $\ov{\Omega}$ and therefore
$\abs{\hb{b}-\nabla\log{\hat{h}}}$ is bounded in $\Omega$ by some positive real constant $\nu$. The proof is completed by applying corollary \ref{cor1}. %and by setting $c_{\perp}=\delta/\lr{2 M^2\nu}$
\end{proof}

\begin{example}(Poincar\'e-like inequality for divergence free Beltrami fields)\label{lemma_}\\
Let $\bol{w}\in C^{\infty}\lr{\ov{\Omega}}$ be a smooth non-vanishing vector field in a smoothly bounded connected domain $\Omega\subset\mathbb{R}^{3}$ such that $\bol{n}\cdot\bol{w}=0$ on $\p\Omega$. 
Assume that $\hb{w}$ is a Beltrami field, i.e. $\hb{b}=\hb{w}\cp\lr{\nabla\cp\hb{w}}=\bol{0}$ in $\Omega$, and $\nabla\cdot\hb{w}=0$ in $\Omega$.
Then:
\begin{equation}
\Norm{\nabla_{\perp} u}\geq \frac{1}{\sup_{\ov{\Omega}}\abs{\bol{x}}}\Norm{u}\,\,\,\,\,\,\,\,\,\forall u\in {\mc{H}}^{\perp}_{0}\lr{\Omega}.
\end{equation}  
%with $c_{\perp}$ a positive real constant depending only on $\Omega$ and $\bol{w}$.
\end{example}

\begin{proof}
Consider the vector field
\begin{equation}
\bol{a}_{\perp}=\hb{w}\cp\lr{\frac{\bol{x}}{2}\cp\hb{w}}~~~~\rm{in}~~~~\Omega.
\end{equation}
We have
\begin{equation}
\nabla\cdot\bol{a}_{\perp}=\frac{1}{2}\lr{\nabla\cdot\bol{x}-\hb{w}\cdot\nabla\lr{\bol{x}\cdot\hb{w}}}=\lr{1+\bol{x}\cdot\hb{b}}=1
\end{equation}
Hence, $\bol{a}_{\perp}$ has the following properties in $\Omega$:
\begin{subequations}
\begin{align}
\bol{a}_{\perp}\cdot\bol{w}&=0,\\
\abs{\nabla\cdot\bol{a}_{\perp}}=1&> 0,\\
\abs{\bol{a}_{\perp}}&\leq\frac{1}{2}\abs{\bol{x}}\leq\nu,
\end{align}
\end{subequations}
where in the last equation we used the fact that $\bol{x}$ is smooth in $\ov{\Omega}$ and therefore
it is bounded in $\Omega$ by some positive real constant $\nu$. 
The proof is completed by applying corollary \ref{cor1} with $2\nu=\sup_{\ov{\Omega}}\abs{\bol{x}}$. %and by setting $c_{\perp}=1/\lr{2\nu}$.
\end{proof}

\section{Existence and uniqueness of solution}

When the conditions of Theorem \ref{PI} are satisfied, $\Norm{u}_{\mc{H}^{\perp}}^{2}=\Norm{u}^{2}+\Norm{\nabla_{\perp}u}^{2}\leq C\Norm{\nabla_{\perp}u}^{2}$ for some positive real constant $C$. 
Then, a new Hilbert space $\tilde{\mc{H}}_{0}$ can be defined with norm $\Norm{\nabla_{\perp}u}$. 
Evidently $\mc{H}_{0}^{\perp}\lr{\Omega}=\tilde{\mc{H}}^{\perp}_{0}\lr{\Omega}$.
With this result, a unique weak solution to the orthogonal Poisson equation can be obtained by application of Riesz's representation theorem.
%In particular, when theorem \ref{PI} applies, one has: 

\begin{theorem} (Existence and uniqueness of solution to the orthogonal Poisson equation) \label{WSol}\\
%Let $\bol{w}\in C^{\infty}\lr{\ov{\Omega}}$ be a smooth vector field in a smoothly bounded connected domain $\Omega\subset\mathbb{R}^{3}$ such that $\bol{n}\cdot\bol{w}=0$ on $\p\Omega$. 
%Assume $\abs{h}=\abs{\JI{w}}\geq \epsilon>0$ in $\Omega$. 
Let $\bol{w}\in C^{\infty}(\tilde{\Omega})$ be a smooth vector field in a bounded domain $\tilde{\Omega}\subset\mathbb{R}^3$ satisfying $\bol{n}\cdot\bol{w}=0$ on $\p\Omega$. 
Assume $\abs{h}=\abs{\JI{w}}\geq \delta>0$ in $\tilde{\Omega}$.
Then, for any $\phi\in L^2\lr{\Omega}$ the orthogonal Poisson equation \eqref{NLapEq}
%with boundary condition $\bol{w}\cp\bol{n}=\bol{0}$ or $\nabla_{\perp}\alpha\cdot\bol{n}=0$ on $\partial\Omega$ 
admits a unique weak solution $u\in {\mc{H}}^{\perp}_{0}\lr{\Omega}$:
\begin{equation}
\lr{u,v}_{\perp}=-\int_{\Omega}{v\,\phi}\,dV\,\,\,\,\,\,\,\,\forall v\in{\mc{H}}_{0}^{\perp}\lr{\Omega}.
\label{ExSol}
\end{equation} 
\end{theorem}

\begin{proof}
In virtue of Theorem \ref{PI}, the linear functional $\int_{\Omega}v\,\phi\,dV$ is bounded. Indeed,
\begin{equation}
\int_{\Omega}v\,\phi\,dV\leq \Norm{v}\Norm{\phi}\leq C\Norm{\nabla_{\perp}v}, 
\end{equation}
with $C$ the positive real constant given in Theorem \ref{PI}. Then, Riesz's theorem can be applied: 
since ${\mc{H}}^{\perp}_{0}\lr{\Omega}$ is a Hilbert space, we can find a unique $u\in {\mc{H}}^{\perp}_{0}\lr{\Omega}$ such that \eqref{ExSol} holds.
\end{proof} 
%\begin{equation}
%\int_{\Omega}{v\phi}\,dV=-\left(v,u\right)_{\perp}\,\,\,\,\,\,\forall v\in \mc{H}^{\perp}_{0}\lr{\Omega}.
%\end{equation}
\noindent Notice that, if $u\in {\mc{H}}^{\perp}_{0}\cap C^{2}\lr{\Omega}$, $u$ is
a classical solution, as can be seen by integration by parts:
\begin{equation}
\left(v,u\right)_{\perp}=\int_{\Omega}{\left(\nabla_{\perp}v\cdot\nabla_{\perp}u\right)}\,dV=
\int_{\partial\Omega}v\nabla_{\perp}u\cdot\bol{n}\,dS
-\int_{\Omega}{v\Delta_{\perp}u}\,dV=
-\int_{\Omega}{v\Delta_{\perp}u}\,dV.
\end{equation}

\section{Concluding Remarks}

In this paper we have studied the boundary value problem associated with a second-order non-elliptic
partial differential operator, the orthogonal Poisson equation, arising in the context of topologically constrained diffusion.
We have shown that a weak unique solution exists whenever it is possible to find
a vector field $\bol{a}_{\perp}$, perpendicular to the constraining vector field $\hb{w}$, with a non-vanishing divergence.
A sufficient conditions for such a vector field to exist is the non-vanishing of the helicity of the constraining vector field.
The solution is an element of a newly introduced Hilbert space, $\mc{H}^{\perp}_{0}$, 
where the inner product is given by the $L^2$ product of the orthogonal gradients of pairs of functions.
Hence, the orthogonal Laplacian is close to the standard Laplacian in that the bilinear form $-\lr{\Delta_{\perp}u,v}$ defines an inner product yielding the Sobolev-like Hilbert space $\mc{H}^{\perp}_{0}$, 
and a Poincar\'e-like inequality is satisfied, $-\lr{\Delta_{\perp}u,u}\geq C\Norm{u}^{2}_{L^2}$.
%characterized by 
%an inner product that combines the $L^2$ inner product of pair of functions with the $L^2$ product of their orthogonal gradients. 

The present results for the non-integrable constraints are in sharp contrast  
with the H-theorem discussed in Ref. \cite{Sato2} for integrable constraints; the integrability/non-integrability of the constraint defines two different classes of the degenerate diffusion equations in three dimensional space. The integrability is related to the helicity of the
constraining vector field. When the helicity vanishes, the constraint foliates the space, leading to non-unique solutions of the orthogonal Poisson equation (with suitable boundary conditions), 
where the multiplicity is expressed in terms of the Casimir invariants. 
On the other hand, non-vanishing helicity makes the orthogonal Laplacian almost like the usual Laplacian, resulting in unique solvability of the boundary value problem.

When we generalize the theory for higher ($>3$) dimension spaces, we need a careful consideration about the kernel of the generalized Poisson matrix; integrable and non-integrable components 
can coexist in the kernel. 
In particular, for an $n$-dimensional antisymmetric operator $\mc{J}$ of rank $2m=n-k$
the complete non-integrability of the $k$-dimensional kernel is guaranteed by the non-vanishing
of the generalized helicity $\bol{h}=\lr{h_{1},...,h_{k}}\neq\bol{0}$, where $h_{i}=\xi_{1}\w ...\w\xi_{k}\w d\xi_{i}$ and the
$\xi_{i}\in T^{\ast}\Omega$ ($i=1,\cdots,k$) are the $k$ 1-forms spanning the kernel of $\mc{J}$.   

\section{Acknowledgment}
We thank Professor Philip J. Morrison for useful discussions on Hamiltonian structures. The work of N.S. was supported by JSPS KAKENHI Grant No. 18J01729, and that of Z.Y. was supported by JSPS KAKENHI Grant No. 17H01177.

%\newpage

\appendix

\section{Topologically constrained diffusion in three spatial dimensions}

The applicability of Fick's laws of diffusion is restricted to systems that live
in a `homogeneous' space, mathematically characterized as the symplectic manifold of canonical phase space.
Topologically constrained systems fail, in general, to exhibit such symplectic structure (see Refs. \cite{Sato1,Sato2}). The loss of canonical phase space directly translates in the non-ellipticity of the
stationary form of the corresponding diffusion equation: each topological constraint represents a spatial direction that is not accessible to the dynamics, and thus the diffusion operator is not sensitive to derivations of the probability density along it. 

In its general form, the dynamics of a three dimensional conservative system is described by the equation:
\begin{equation}
\bol{v}=\bol{w}\cp\nabla H,
\end{equation}
where $\bol{v}=\dot{\bol{x}}$ is the velocity, $\bol{w}=\bol{w}\lr{\bol{x}}$ the constraining vector field, and $H=H\lr{\bol{x}}$ the Hamiltonian function. Both $\bol{w}$ and $H$ are assumed smooth in their domain.
This system is conservative because it preserves the energy $H$. 
$\bol{w}$ is a constraining vector field because dynamics always obeys the constraint:
\begin{equation}
\bol{w}\cdot\bol{v}=0.
\end{equation}
The constraint is inetgrable if the Frobenius itegrability condition (see Ref. \cite{Frankel}) for the vector field $\bol{w}$ is satisfied:
\begin{equation}
\bol{w}\cdot\nabla\cp\bol{w}=0.
\end{equation}
The quantity $h=\bol{w}\cdot\nabla\cp\bol{w}$, which does not vanish in general, is called the helicity of $\bol{w}$.

By neglecting deterministic terms in the Hamiltonian $H$, and setting $\nabla H=\bol{\Gamma}$, where $\bol{\Gamma}$ is three dimensional Gaussian white noise, one arrives at the stochastic differential equation:
\begin{equation}
\bol{v}=\bol{w}\cp\bol{\Gamma}.\label{SDE}
\end{equation} 
The diffusion equation for the probability density $u=u\lr{\bol{x}}$ corresponding to \eqref{SDE} is then (see Ref. \cite{Sato2}):
\begin{equation}
\frac{\p u}{\p t}=\frac{1}{2}\nabla\cdot\left[\bol{w}\cp\lr{\nabla\cp u\,\bol{w}}\right].\label{FPE3D}
\end{equation}
Assuming that $\abs{\bol{w}}=w\neq 0$, the stationary form of \eqref{FPE3D} reads:
\begin{equation}
\Delta_{\perp}u+\lr{\hb{b}+3\nabla_{\perp}\log w^2}\cdot\nabla_{\perp}u+\lr{\nabla_{\perp}\log w^2\cdot\hb{b}+\hat{\mf{B}}+\frac{1}{2\, w^2}\Delta_{\perp}w^2}u=0.\label{FPE3DStat}
\end{equation}
Here, we introduced the normalized vector field $\hb{w}=\bol{w}/w$, and the quantities $\hb{b}=\hb{w}\cp\lr{\nabla\cp\hb{w}}$ and $\hat{\mf{B}}=\nabla\cdot\hb{b}$, which are called field force and field charge  respectively (see also Ref.\cite{Sato2} for the definition of field force and field charge). Furthermore, $\nabla_{\perp} f=\hb{w}\cp\lr{\nabla f\cp\hb{w}}$ is the normal component of the gradient $\nabla f$ of the function $f$ with respect to the constraining vector field $\hb{w}$, and $\Delta_{\perp}f=\nabla\cdot\lr{\nabla_{\perp}f}$ is the \textit{orthogonal Laplacian} of the function $f$. 
The non-ellipticity of equation \eqref{FPE3DStat} is manifest in that it does not invole
derivatives of the probability density $u$ in the direction parallel to $\hb{w}$.

\end{normalsize}

\end{document}